\documentclass[10pt]{amsart}
\usepackage{amssymb,amsmath,amscd,amsthm}
\usepackage{mathptmx}
\usepackage{palatino}
\usepackage{bbm}
\usepackage{eucal}
\usepackage{epstopdf}
\usepackage{amsxtra}
\usepackage{graphicx}
\usepackage{latexsym}
\usepackage{setspace}
\usepackage{mdwlist}
\input xy
\xyoption{all}
\usepackage{blkarray}
\usepackage{hhline}
\usepackage{multirow,bigdelim}
\usepackage{algorithm,algorithmic}
\usepackage{paralist}
\usepackage{graphicx}
\usepackage{caption}
\usepackage{subcaption}
\usepackage[T1]{fontenc}
\usepackage[table]{xcolor}

\usepackage{mathrsfs}
\usepackage{amsbsy}
\usepackage{bbm}
\usepackage{mathptmx}
\usepackage[scaled=0.92]{helvet}
\usepackage{courier}
\usepackage{fancyhdr}
\usepackage{fancybox}
\cornersize{1.6cm}
\RequirePackage{HensLaTexPack3}

\newtheorem{proposition}[theorem]{Proposition}

\newcommand{\fl}{\mathbf{F}}
\newcommand{\gl}{\mathbf{G}}
\newcommand{\il}{\mathbf{I}}
\renewcommand{\bl}{\mathbf{B}}
\newcommand{\cl}{\mathbf{C}}
\renewcommand{\sl}{\mathbf{S}}
\newcommand{\pl}{\mathbf{P}}

\newcommand{\hl}{\mathbf{H}}

\newcommand{\field}{{\mathds k }}
\newcommand{\Eirene}{\href{http://gregoryhenselman.org/eirene.html}{\sc Eirene} }
\newcommand{\op}{{^{\mathrm op}}}
\usepackage{dsfont}

\makeatletter

\newcommand{\Rmnum}[1]{\expandafter\@slowromancap\romannumeral #1@}
\makeatother

\begin{document}

\title{Matroid Filtrations and Computational Persistent Homology}
\author{Gregory Henselman}
\address{Department of Electrical/Systems Engineering,
        University of Pennsylvania,
        Philadelphia PA, USA}
\email{grh@princeton.edu}

\author{Robert Ghrist}
\address{Departments of Mathematics and Electrical/Systems Engineering,
        University of Pennsylvania,
        Philadelphia PA, USA}
\email{ghrist@math.upenn.edu}

\begin{abstract}
This technical report introduces a novel approach to efficient
computation in homological algebra over fields, with particular
emphasis on computing the persistent homology of a filtered topological
cell complex. The algorithms here presented rely on a novel relationship
between discrete Morse theory, matroid theory, and classical matrix
factorizations. We provide background, detail the algorithms, and
benchmark the software implementation in the \Eirene  package.
\end{abstract}

\date{\today}
\maketitle

\section{Introduction}
\label{sec:intro}

This paper reports recent work toward the goal of establishing a set
of efficient computational tools in homological algebra over a field
\cite{henselmanMFA16}. Our principal motivation is to
improve upon the foundation built in \cite{CGNDiscrete13} and develop
and implement efficient algorithms for computing cosheaf
homology and sheaf cohomology in the cellular setting.  One
can specialize to constructible cosheaf homology over the
real line, which includes persistent homology as a particular instance.  As this incidental
application is of particular current relevance to topological data analysis
\cite{CarlssonTopology09,EHComputational10,GhristElementary14,GPC+Clique15,OudotPersistence15,RobinsonTopological14},
we frame our methods to this setting in this technical report.

We present a novel algorithm for computing persistent homology of a
filtered complex quickly and memory-efficiently. There are three
ingredients in our approach.
\begin{enumerate}
\item {\bf Matroid theory.} Our algorithms are couched in the language
of matroids, both for elegance and for efficiency. This
theory permits simple reformulation of standard notions of persistence
[filtrations, barcodes] in a manner that incorporates matroid concepts
[rank, modularity, minimal bases] and the concomitant matroid algorithms.
\item {\bf Discrete Morse theory.} It is well-known \cite{BauerPersistence11,Bauer2014,DKM+Coreduction11,KozlovCombinatorial08,MNMorse13}
that discrete Morse theory can be used to perform a simplification of
the complex while preserving homology. We build on this to use the
Morse theory to avoid construction and storage of the full boundary
operator, thus improving memory performance.
\item {\bf Matrix factorization.} Combining matroid and Morse
perspectives with classical matrix factorization theorems allows for
both more efficient computation as well as the ability to back-out
representative generators of homology as chains in the {\em original}
input complex, a capability of significant interest in topological data analysis.
\end{enumerate}

This technical report is intended for a reader familiar with homology,
persistent homology, computational homology, and (elementary) discrete
Morse theory: as such, we pass over the usual literature review and
basic background.  Since (in the computational topology community)
matroid theory is much less familiar, we provide enough background  in \S\ref{sec:background}
to explain the language used in our algorithms. Matrix reformulation of persistence
and subsequent algorithms are presented in 
\S\ref{sec:modularfiltrations}-\S\ref{sec:acyclicrelations}. 
Implementation and benchmarking appears in \S\ref{sec:experiments}.

\section{Background}
\label{sec:background}

For the purpose of continuity let us fix a ground field $\field$.  Given a $\field$-vector space $V$ and a set inclusion $S \into V$, we write $\nr{S}$ for the subspace spanned by $S$ and $V/S$ for the quotient of $V$ by $\nr{S}$.  Where context leaves no room for confusion we will identify the elements of $V$ with their images under quotient maps, making it possible, for example, to speak unambiguously of the linear independence of $T \su V$ in $V / S$.  In the case where multiple elements of $T$ map to the same element of $V/S$ under the quotient projection, we regard $T$ as a multiset of $V/S$.  Given an arbitrary set $E$ and singleton $\{j\}$, we write $E+j$ for $E \cup \{j\}$, and given any binary relation $R$ on $E$ we write $R\op$ for $\{(t,s): (s,t) \in R\}$.

The matroid theory required to read and understand this text is minimal -- full details may be found  in a number of excellent introductory texts, e.g.\ \cite{orientedmatroids93,OxleyMatroid11,PSCombinatorial98,truemperMatroidDecomposition}.  The reader should be aware that while, for historical reasons, many authors implicitly assume their {matroids} to be finite, there exist several notions of a matroid on an infinite ground set.  For economy we will write \emph{matroid} for  \emph{finitary matroid of finite rank}.

A (finite-rank, finitary) {matroid} $M$ consists of a pair $(E, \il)$, where $E$ is a set, $\il \su 2^E$ is a family of subsets, and  (i) $J \in \il$ whenever $I \in \il$ and $J \su I$, and (ii) if $I,J  \in \il$ with $|I|< |J|$, then there exists $j \in J - I$ such that $I + j \in \il$.  The sets $E$ and $\il$, sometimes written $E(M)$ and $\il(M)$ to emphasize association with $M$, are called the \emph{ground set} and \emph{independence system}, respectively, of $M$.  The elements of $\il$ are  \emph{independent sets} and those of $2^E - \il$ are  \emph{dependent}.  Maximal-under-inclusion independent sets are  \emph{bases} and minimal-under-inclusion dependent sets are  \emph{circuits}.    It follows directly from the axioms that a singleton $\{e\}$ forms a circuit if and only if the replacement set
$$
R(e \to B) = \{b \in B: B -b+ e \in \bl(M) \}
$$
is empty for some basis $B$.  The family of all bases  is denoted $\bl(M)$, and that of all circuits is denoted  $\cl(M)$.

The structure imposed by condition (ii)  gives life to a remarkable branch of combinatorics, with fundamental ties to discrete and algebraic geometry,  optimization, algebraic and combinatorial topology, graph theory, and analysis of algorithms.  It also underlies a number of the useful properties associated with finite-dimensional vector spaces, many of which, including dimension, have matroid-theoretic analogs.
To illustrate, consider a matroid found very often in literature:  that of a subset $U$ in a finite-dimensional, $\field$-linear vector space $V$.  To every such $U$ corresponds an independence system $\il$ consisting of the $\field$-linearly independent subsets of $U$, and it is an elementary exercise to show that the pair $M(U) = (U, \il)$ satisfies the axioms of a finite-rank, finitary matroid.     \\

\noindent \emph{Zero}  We say $e \in E$ is a \emph{zero element} if $\{e\} \in \cl(M)$, or, equivalently, if $R(e \to B)$ is empty.  Clearly  $u \in U$ is a zero in $M(U)$ if and only if $u = 0 \in V$. \\

\noindent \emph{Rank}  Axiom (ii) can be shown to imply that $|I| = |J|$ for every pair of maximal-with-respect-to-inclusion independent subsets of an arbitrary $S \su E$.  This number, denoted $\rk(S)$, is called the \emph{rank} of $S$.  By convention $\rk(M) = \rk(E)$.  If $M = M(U)$, then $\rk(S) = \dim \nr{S}$.  \\

\noindent \emph{Closure}  The \emph{closure} of  $S \su E$, written either $\ov S$ or $cl(S)$, is the discrete analog of a subspace generated by $S$ in $V$: formally, $\ov S$ is set of all $e \in E$ such that $\{e\} = C - S$ for some $C \in \cl(M)$.  If $M = M(U)$, then $\ov S=U \cap \nr{S}$.  A \emph{flat} is a subset that equals its closure.  \\

\noindent \emph{Deletion}   If $S \su E$ then the family $\il = \{T \in \il(M): T \su E - S\}$ is the independence system of a matroid $M-S$ on ground set $E-S$, called the \emph{deletion} of $M$ by $S$.   The \emph{restriction} of $M$ to $S$, written $M|S$, is the deletion of $M$ by $E-S$.  If $M = M(U)$ then $M|S = M(S)$.\\

\noindent \emph{Contraction}  The \emph{contraction of M by S}, denoted $M/S$, is the matroid on ground set $E - S$ with independence system $\{T \in \il(M): T \cup J \in \il \foral J \in \il \cap 2^S\}.$    If $I$ is any maximal independent subset of $S$, then it may be shown that  $\il(M/S) = \{J \su E - S: J \cup I \in \il(M)\}$.  As an immediate corollary, $\rk(M/S) = \rk(M) - \rk(S)$.  If $M = M(U)$ then 
$$
M/S = (E - S, \{I \su S: \tm{multiset $I$ is linearly independent in $V/\nr{S}$}
\}).
$$
\vspace{-0.25cm}

\noindent \emph{Minor}   Matroids obtained by sequential deletion and contraction operations are \emph{minors} of $M$.  It can be shown that deletion and contraction commute, so that every minor may be expressed $(M-S)/T$ for some $S$ and $T$.  Where context leaves no room for confusion we write $S/T$ for  $(M|S)/T$.     \\

\noindent \emph{Representation}  A ($\field$-linear) representation of $M$ is a map $\phi: E \to \field^r$ such that $I \in \il$ if and only if $\phi(I)$ is linearly independent.  Matroids that admit $\field$-linear representations are called $\field$-linear.  It is common practice to identify a representation $E \to \field^r$ with the matrix $A \in \field^{ \{1, \ld, r\} \times E }$ such that $A[\;: \;,e] = \fk(e)$ for all $e \in E$.  In general, we will not distinguish between $A$ and the associated $E$-indexed family of column vectors.   Given $B \in\bl(M)$, we say $\fk$ has \emph{$B$-standard form} if $\fk(B)$ is the basis of standard unit vectors in $\field^r$.   Clearly $B$-standard representations of $M(\field^r)$ are in 1-1 correspondence with $GL_r(\field)$.\\

We note a few facts regarding discrete optimization.   Let $f: E \to \R$ be any weight function, $\sl \su e^E$ be any family of subsets, and $\bl \su \sl$ be the family of maximal-under-inclusion elements of $\sl$. The elements of $\tm{argmax}_{B \in \bl } f(B)$, where $f(B)=\sum_{b \in B} f(b)$, are of fundamental importance to a number of fields in discrete geometry and combinatorics, and the problem of finding the maximal $B$ given $E, f$, and an oracle to determine membership in $\sl$ has been vigorously studied since the mid-20th century. This problem,  NP-hard in general, admits a highly efficient greedy solution in the special case where $(E, \sl)$ constitutes a matroid.  In fact, this characterizes finite matroids completely. The following is classical:

\begin{proposition}
\label{prop:matroidcharacterization}
Let $E$ be a finite set and $\sl$ be a subset of $2^E$ closed under inclusion.  Then $\sl$ is the independence system of a matroid on ground set $E$ if and only if Algorithm \ref{alg:matroidalgorithm} returns an element of $\tm{argmax}_\sl f$ for arbitrary $f: E \to \R$.
\end{proposition}
Proposition \ref{prop:matroidcharacterization} affords a number of efficient proof techniques in the theory of finite matroids, as will be seen.  As $\tm{argmin}_{\sl} f = \tm{argmax}_{\sl}(-f)$, Algorithm \ref{alg:matroidalgorithm} also provides a greedy method for obtaining minimal bases, which will be of primary concern.

\begin{algorithm}
\caption{Greedy algorithm for maximal set-weight}
\label{alg:matroidalgorithm}
\begin{algorithmic}[1]
\STATE Label the elements of $E$ by $e_1, \ld, e_n$, such that $f(e_{i+1}) \le f(e_i)$.
\STATE $S:=\emptyset$
 \FOR{$i = 1, \ld, n$}
 		\IF {$S + s \in \sl$}
			\STATE $S \leftarrow S + s$
		\ENDIF
 \ENDFOR
\end{algorithmic}
\end{algorithm}

Our final comments concern the relationship between a pair of $f$-minimal bases.  Up to this point, the terms and observations introduced have been standard elements of matroid canon.  In the following sections we will need two new notions, namely those of a  flat function and an $f$-triangular matrix.

Given $S, T \su E$ and  $f: E \to \Z$, say that $A \in \field^{S \times T}$ is \emph{$f$-upper triangular} if   $f[s] \le f[t]$ whenever $A[s,t]$ is nonzero.  The product of two $f$-upper triangular matrices is again upper-triangular, as are their inverses.  We say $A$ is \emph{$f$-lower triangular} if it is $g$-upper triangular for $g = -f$, and $f$-diagonal if it is  $f$-upper and $f$-lower triangular.  We say $f$ is \emph{flat} if the inverse image $f\inv (-\infty, t]$ is a flat of $M$ for every choice of $t$.

\begin{lemma}
\label{lem:besttriangle}
Let $f:E\to\Z$ be a flat weight function; $B, F$ bases of $M$; and  $R_B$ the unique element of $\{0,1\}^{B \times E}$ such that
\[
R_B[b,e] = \begin{cases}
1 & \; \tm{ if } \; B - b + e \in \bl(M) \\
0 & \; \tm{ otherwise}.
\end{cases}
\]
If $B$ is $f$-minimal, then $F$ is $f$-minimal if and only if $R_B[B,F]$ is $f$-upper triangular.
\end{lemma}

\begin{lemma}
Let $M$ be a linear matroid and $f$ a flat weight function on $M$.  If $F$ is an $f$-minimal basis and $L \in GL_F(\field)$, $R \in GL_E(\field)$ are $f$-upper triangular  then a basis $B$ is $f$-minimal if and only if $LA[F, B]R$ is $f$-upper triangular for every $F$-standard representation $A$.\end{lemma}

\begin{lemma}
\label{lem:stillgotit}
Let $f$ be a flat weight function on $M(\field^E)$, and $\delta \in \field^{E \times E}$ be the Dirac delta.   If $S$ is finite and $L,R$ are $f$-upper triangular elements of $GL_E(\field)$ and $GL_S(\field)$, respectively, then the columns of $\delta|_{E \times S}$ contain an $f$-minimal basis iff those of $L \delta|_{E \times S} R$ contain one, also.
\end{lemma}

\section{Modular filtrations}
\label{sec:modularfiltrations}

To make precise the place of linear persistence in discrete optimization,  we introduce one further notion, that of modular filtration. Much work in the recent field of homological data analysis may be viewed as a treatment of the relationship between two or more filtrations on a vector space.  By analogy, we define  a {filtration} $\fl$ of matroid $M$ to be a nested sequence of flats $\emptyset = \fl_0 \su \cd \su \fl_L = E$.  The \emph{characteristic function} of $\fl$ is the  map $\chi_\fl: E \to \Z_{\ge 0}$ such that $\fl_k = \chi_\fl \inv \{0, \ld, k\}$ for all $k$.  A pair of filtrations $(\fl, \gl)$ is \emph{modular} if
\begin{align}
\rk(\fl_i \cap \gl_j) + \rk(\gl_i \cup \gl_j) = \rk(\fl_i) + \rk(\gl_j)
\end{align}
for all $i$ and $j$.  Modularity is of marked structural significance in the general theory of matroids, and offers a powerful array of tools for the analysis of filtrations.

\begin{proposition}
\label{prop:unifyingtheorem}
If $\fl, \gl$ are filtrations of $M$, then the following are equivalent.
\begin{enumerate}
\item $(\fl, \gl)$ is modular.
\item There exists a basis of minimal weight with respect to both $\chi_\fl$ and $\chi_\gl$.
\item There exists a basis $B$ such that
\begin{align}
B \cap (\fl_i \cap \gl_j) \in \bl(\fl_i \cap \gl_j)  \notag
\end{align}
for all $i,j$.
\end{enumerate}
\end{proposition}

The proof is organized as follows.  Equivalence of (2) and (3) follows from Lemma \ref{lem:intersectioncharacterization} below.\footnote{Lemma \ref{lem:intersectioncharacterization} yields, moreover, a highly efficient means of checking independence over large families of intersections, and will be used implicitly throughout.}  It is readily seen that (1) follows from (3), as the latter implies $\rk(\fl_i \cap \gl_j) = |B \cap \fl_i \cap \gl_j|$, thereby reducing the identity that defines modularity  to
$$
|B \cap \fl_i \cap \gl_j| = |B \cap \fl_i| + |B \cap \gl_j| - |B \cap \fl_i \cap \gl_j|.
$$
It remains to be shown, therefore, that (1) implies either (2) or (3).  Given our special interest in minimal bases, a constructive proof of the former is to be desired.  

\begin{lemma}
\label{lem:intersectioncharacterization}
A basis $B$ has minimal weight with respect to $\chi_\fl$ and $\chi_\gl$ if and only if
\begin{align}
B \cap (\fl_i \cap \gl_j) \in \bl(\fl_i \cap \gl_j)
\label{eq:modularbasisintersection}
\end{align}
for all $i, j$.
\end{lemma}

\begin{proof}
If $B$ satisfies \eqref{eq:modularbasisintersection} for all $i,j$ then minimality with respect to $\fl$ follows from an application of Lemma \ref{lem:besttriangle} to the family of intersections $B \cap \fl_i \cap E$.  (Why?  If $B \cap \fl_i$ is a basis for $\fl_i$, then $\rk(\fl_i) = |B \cap \fl_i|$.  If, therefore, $C = B - b + e \in \bl(M)$ for some $b \notin \fl_i$ and $e \in  \fl_i$, then $C \cap \fl_i$ is an independent subset of $\fl_i$ of order $|B \cap \fl_i|+1 = \rk(\fl_i) +1$, a contradiction.  Therefore $R_B(b,e) = 0$ when $e \in \fl_i$ and $b \notin \fl_i$.  Apply Lemma \eqref{eq:modularbasisintersection}, with $f = \chi_\fl$).   Minimality with respect to $\bl$ follows likewise.  If on the other hand $S = B \cap (\fl_i \cap \gl_j) \notin \bl(\fl_i \cap \gl_j)$ for some $i,j$, then there exists $s \in  \fl_i \cap \gl_j$ such that $S + s \in \il(M)$.  The fundamental circuit of $s$ with respect to $B$ intersects $B - S$ nontrivially, hence $B-b+s \in \bl(M)$ for some $b \in B - S$.  Either $\chi_\fl(s) < \chi_\fl(b)$ or $\chi_\gl(s) < \chi_\gl(b)$, and the desired conclusion follows.
\end{proof}

\begin{proof}[Proof of Proposition \ref{prop:unifyingtheorem}]
As already discussed, it suffices to show (1) implies (2).  Therefore assume $(\fl, \gl)$ is modular, and for each $i$ fix  a $\chi_\gl$-minimal
$$
B_i \in \bl(\fl_i / \fl_{i-1}).
$$
The union $B= \cup_i B_i \in \bl(M)$ forms a $\chi_\fl$-minimal basis in $M$, so it suffices to show $B$ is  minimal with respect to $\gl$.  Since $|B \cap G_j| \le \rk(G_j)$, we may do so by proving $|B \cap \gl_j| \ge \rk(\gl_j)$ for all $j$.   Therefore let $i$ and $j$ be given,  fix $S_{i-1} \in \bl(\fl_{i-1} \cap \gl_j)$ and extend $S_{i-1}$ to a basis $S_i$ of $\fl_i \cap \gl_j.$

Modularity of $(\fl_i, \gl_j)$ implies 
$$
\rk(\gl_j/\fl_i) = \rk(\gl_j \cup \fl_i)-\rk(\fl_i) = \rk(\gl_j) - \rk(\gl_j \cap \fl_i) = \rk(\gl_j / (\gl_j \cap \fl_i)) ,
$$
so that, for all $i$ and $j$,
$$
\rk(\gl_j / \fl_i) = \rk(\gl_j / (\fl_i \cap \gl_j)) = \rk(\gl_j / S_i) \le \rk(\gl_j / (S_i \cup \fl_{i-1})) \le \rk(\gl_j / \fl_i).
$$
As the quantities on either side are identical, strict equality holds throughout.  Thus the second equality below.
\begin{align}
\rk(\gl_j / S_{i-1}) - \rk(S_i / S_{i-1}) & = \rk(\gl_j / S_i)    \notag \\
 & = \rk(\gl_j / (S_i \cup \fl_{i-1}) = \rk(\gl_j / \fl_{i-1}) - \rk(S_i / \fl_{i-1}). \notag
\end{align}
A comparison of left- and right-hand sides shows $\rk(S_i/ S_{i-1}) = \rk(S_i / \fl_{i-1})$.  By construction the set $T_{ij} = S_i - S_{i-1} \in \gl_j \cap \fl_i - \fl_{i-1}$ forms a basis in $S_i / S_{i-1}$, hence
$$
    |T_{ij}| =  \rk(S_i / S_{i-1}) =\rk(S_i / \fl_{i-1})  =  \rk(T_{ij} / \fl_{i-1}).
$$
Thus $T_{ij} \in \il(\fl_i / \fl_{i-1})$, so $|B_i \cap \gl_j| \ge |T_{ij}|$.  A second and third application of modularity provide the second and third equalities below,
\begin{align*}
|T_{ij}| & = \rk(\gl_j / S_{i-1}) - \rk(\gl_j/S_i)  \\
&= \rk(\gl_j/\fl_{i-1}) - \rk(\gl_j/\fl_{i}) \\
&= \rk(\gl_j \cap \fl_i) - \rk(\gl_j \cap \fl_{i-1})
\end{align*}
whence $|B_i \cap \gl_j| \ge   \rk(\gl_j \cap \fl_i) - \rk(\gl_j \cap \fl_{i-1})$.  Summing over $i$ yields
\begin{align}
|B \cap \gl_j| \ge \rk(\gl_j)  \notag
\end{align}
which was to be shown.
\end{proof}

\section{Linear Filtrations}
\label{sec:linearfiltrations}

We now specialize to the case of linear filtrations.  Recall that a  linear filtration of a finite-dimensional vector space $V$ is a nested sequence of subspaces $\emptyset = \fl_0 \su \cd \su \fl_{L} = V$.  Clearly, the linear filtrations of $V$ are exactly the matroid-theoretic filtrations of $M(V)$.   Since $\dim \fl_i + \dim \gl_j = \dim \fl_i \cap \gl_j + \dim \nr{\fl_i \cup \gl_j}$ when $\fl$ and $\gl$ are linear, Proposition \ref{prop:unifyingtheorem} implies that to every such pair corresponds a basis $B  \in \bl(M)$ such that $\fl_i \cap \gl_j = \nr{B \cap \fl_i \cap \gl_j}$ for all $i,j$.  We are interested in computing  $B$ when $V = \field^r$.

Assume that the data available for this pursuit are (1) an $\fl$-minimal basis $F$, (2) a finite set $G$ containing a $\chi_{\gl}$-minimal basis, and (3) oracles to evaluate $\chi_\fl$ on $F$ and $\chi_\gl$ on $G$.  These resources are commonly accessible for computations in persistent homology, as described in the following section.  One can assume, further, that $F$ is the set of standard unit vectors in $V = \field^F$.  If such is not the case one can solve the analogous problem for an $F$-standard representation $\fk \in GL(V)$, and port the solution back along $\fk \inv$.  Our strategy will be to construct a $(\chi_\fl, \chi_\gl)$-minimal basis from linear combinations of the vectors in $G$.  A few pieces of notation will aid its description.

Given a matrix $A \in \field^{F \times G}$, write $\tm{supp } s$ and $\tm{supp } t$ for the supports of row and column vectors indexed by, $s$ and $t$, respectively. Say that a  relation $R$ on $E$ \emph{respects} $f:E \to \R$, or that $R$ is an \emph{$f$-relation}, if  $f(s) \le f(t)$ whenever $(s,t) \in R$.
Given a $\chi_\fl$-linear order $<_\fl$ and a $\chi_\gl$-linear order $<_\gl$, define $\pl$, $S$, and $T$ by
\begin{align*}
\pl &= \{(f,g) :  f = \tm{max}_{<_\fl} \supp g , \; g = \tm{min}_{<_\gl} \supp f \} \\
S & = \{f : (f,g) \in \pl(A, <_\fl, <_\gl) \tm{ for some } g\} \\
T & = \{g : (f,g) \in \pl(A, <_\fl, <_\gl) \tm{ for some } f\}
\end{align*}
and $L$, $R$, $X$, $Y$, $Z$, and $*$ by
\begin{align*}
L = \begin{array}{r|c|c|}
\multicolumn{1}{c}{}& \multicolumn{1}{c}{S} &  \multicolumn{1}{c}{ S^c } \\  \cline{2-3}
 S & X\inv  & 0  \\   \cline{2-3}
 S^c & -X\inv Z  & I \\ \cline{2-3}
\end{array}
\hspace{1cm}
A = \begin{array}{r|c|c|}
\multicolumn{1}{c}{}& \multicolumn{1}{c}{T} &  \multicolumn{1}{c}{ T^c } \\  \cline{2-3}
 S & X  & Y  \\   \cline{2-3}
 S^c & Z  & *\\ \cline{2-3}
\end{array} 		
\hspace{1cm}
R = \begin{array}{r|c|c|}
\multicolumn{1}{c}{}& \multicolumn{1}{c}{T} &  \multicolumn{1}{c}{ T^c } \\  \cline{2-3}
 T & I  & -X\inv Y  \\   \cline{2-3}
 T^c & 0  & I \\ \cline{2-3}
\end{array}
\end{align*}
where  $S^c = F - S$ and $T^c = G - T$.   It can be helpful to regard $\pl$  as the Pareto-optimal frontier of $\tm{supp } A = \{(s,t): A[s,t] \neq 0\}$ with respect to $<_\gl$ and the inverse-order  $<_\fl\op$.   See Figure \ref{fig:paretomatrix}.
			
Lemma \ref{lem:atomstep} now holds by construction.
\begin{lemma}
\label{lem:atomstep}
If $\pl$, $S$, and $T$ are as above, then
\begin{enumerate}
\item $L$ is $\chi_\fl$-lower-triangular,
\item $R$ is $\chi_\gl$-upper-triangular, and
\item The product $LAR$ satisfies
$$
LAR = \begin{array}{r|c|c|}
\multicolumn{1}{c}{}& \multicolumn{1}{c}{T} &  \multicolumn{1}{c}{ T^c } \\  \cline{2-3}
 S & I  & 0  \\   \cline{2-3}
 S^c & 0  & *\\ \cline{2-3}
\end{array} 	
$$	
for some block-submatrix $*$ of rank equal to $ \rm{rank}\ A - |S|$.
\end{enumerate}
\end{lemma}

Where convenient we  will write $L(A,<_\fl,<_\gl)$, $R(A,<_\fl,<_\gl)$, and $\pl(A,<_\fl,<_\gl)$ for $L,R,$ and $\pl$ to emphasize association with their genera.  In the description of Algorithm \ref{alg:paretoalgorithm} we will write  $L_t$ and $R_t$  for values of  $L$ and $R$ associated to $A_t$.  

\begin{algorithm}
\caption{Matrix reduction}
\label{alg:paretoalgorithm}
\begin{algorithmic}[1]
 \FOR{$t = 1, \ld, r$}	
		\STATE $A_{t+1} \leftarrow L_tA_tR_t$
		\IF{$|A_{t+1}|_\infty = r$}
			\STATE break
		\ENDIF
 \ENDFOR
\end{algorithmic}
\end{algorithm}

The proof of Lemma \ref{lem:integrateatoms} follows easily from Lemma \ref{lem:atomstep}:

\begin{lemma}
\label{lem:integrateatoms}
Stopping condition $|A_{t+1}|_\infty = r$ holds  for some $1 \le t \le r$.  The matrix products $L = L_t \cdot \cdots \cdot L_1 $ and $R = R_1 \cdot \cdots R_t $ are $\chi_\fl$-upper-triangular and $\chi_\gl$-lower-triangular, respectively.
\end{lemma}

\begin{corollary}
The columns of $L\inv$ form an $(\fl,\gl)$-minimal   basis in $M(V)$.
\end{corollary}
\begin{proof}
If $L$ and $R$ are as in Lemma \ref{lem:integrateatoms}, then up to permutation and nonzero-scaling $LAR = [\; I \; | \; 0 \; ]$.  By Lemma \ref{lem:stillgotit}, $AR = [\; L \inv \; | \; 0 \; ]$ contains  a $\gl$-minimal basis, perforce the set of column vectors  of $L\inv$.  Since in addition $L\inv$ is $\chi_\fl$-upper-triangular, the associated basis is also $\chi_\fl$-minimal.
\end{proof}

It can be shown that $L(A,<_\fl,<_\gl) = L(AR,<_\fl,<_\gl)$, where $R = R(A,<_\fl,<_\gl)$.  Consequently, the matrices $R_t$ in Algorithm \ref{alg:paretoalgorithm} need not play any role whatever in the calculation of $L \inv$:  one must simply modify the stopping criterion to accommodate the fact that $LA$ may have more nonzero coefficients than $LAR$.  See Algorithm \ref{alg:paretoalgorithm_light}.

\begin{algorithm}
\caption{Light matrix reduction}
\label{alg:paretoalgorithm_light}
\begin{algorithmic}[1]
\STATE Initialize $A_1 = A$
 \WHILE{$|\pl(A_{t})| < r$}	
		\STATE $A_{t+1} \leftarrow L_tA_t$, \; $t \leftarrow t+1$
 \ENDWHILE
\end{algorithmic}
\end{algorithm}

\begin{figure}
\vspace{0.8cm}
\includegraphics[width=1.3in]{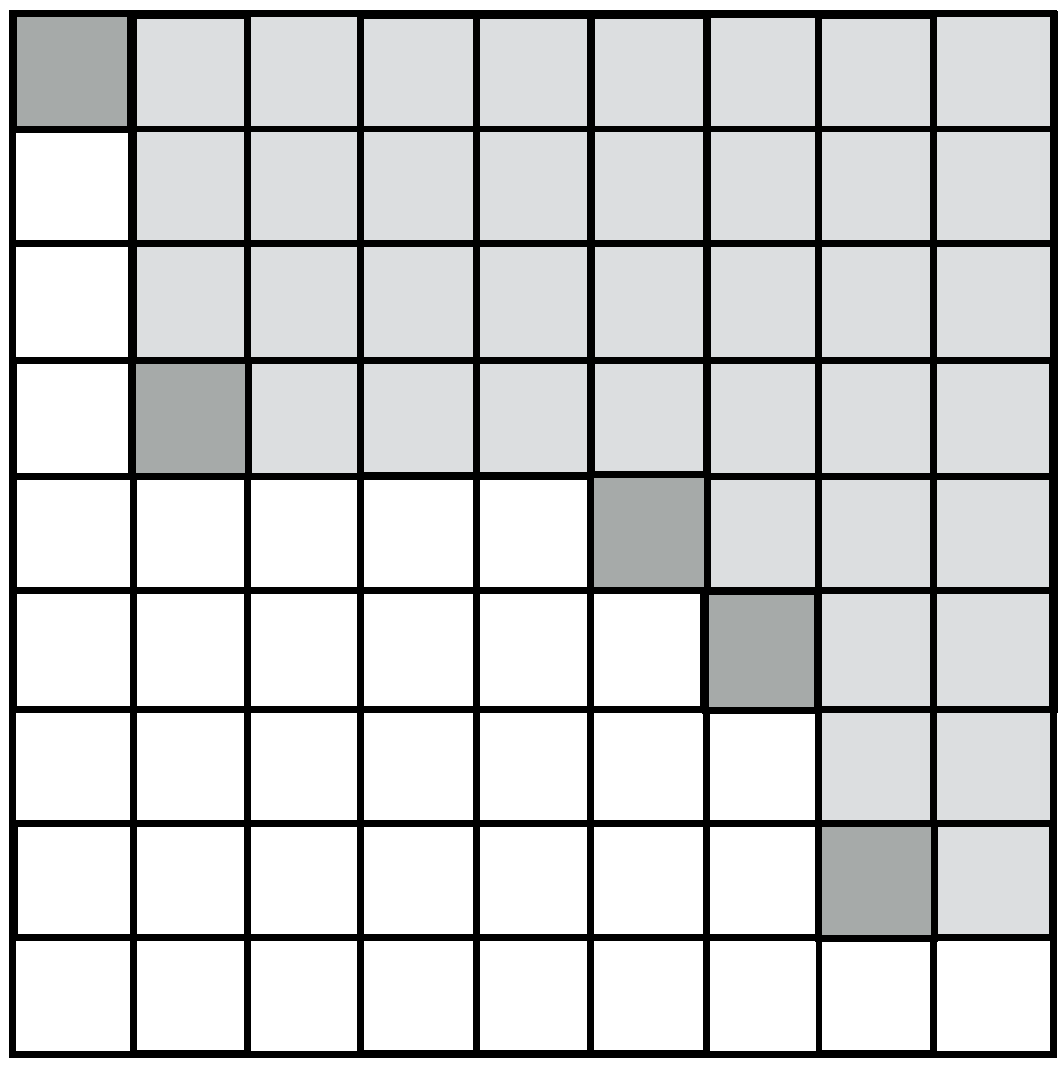}
\hspace{1.75cm}
\includegraphics[width=1.3in]{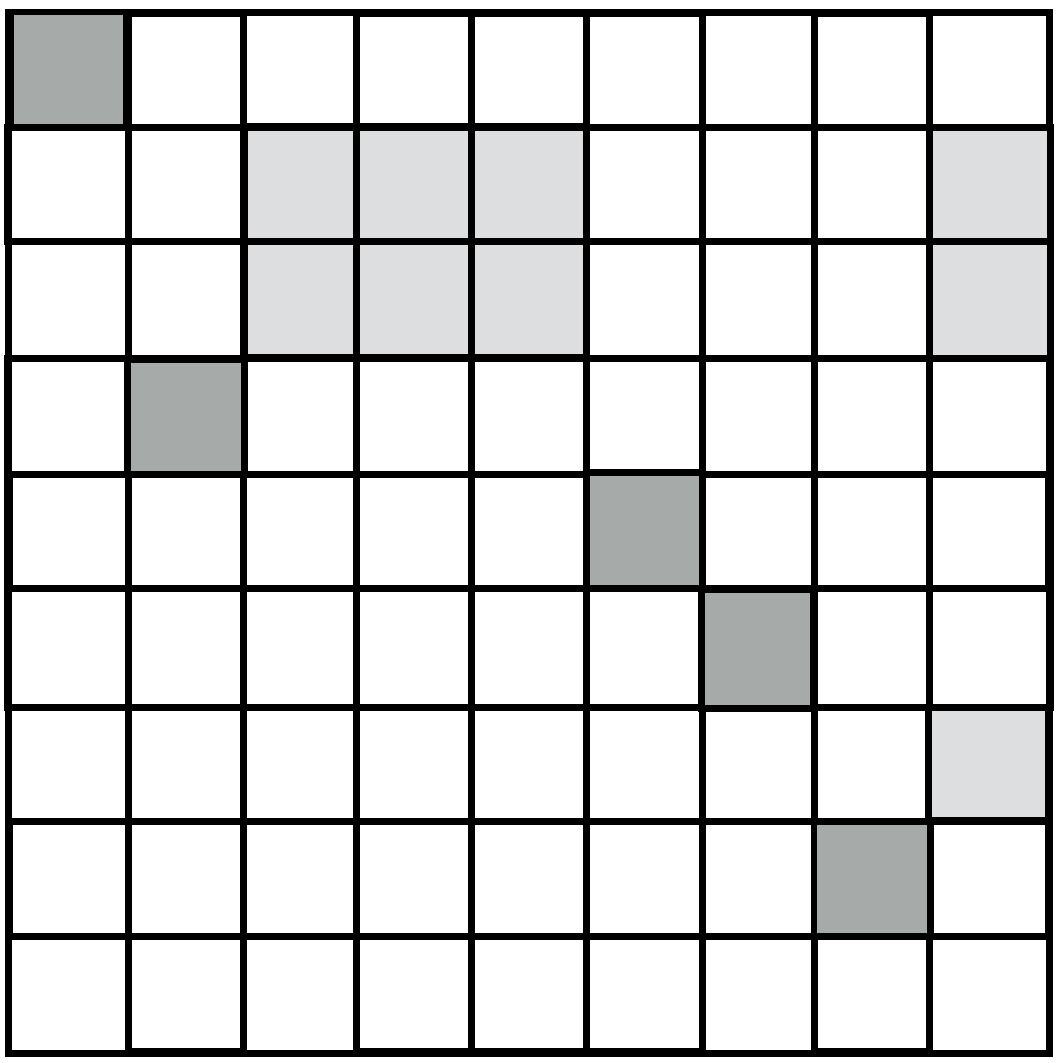}
\vspace{0.5cm}
\captionsetup{width=0.9\textwidth}
\caption{\emph{Left:}  The sparsity pattern of a matrix $A$ with rows and columns in ascending $<_\fl$ and  $<_\gl$ order.  Nonzero coefficients are shaded in light or dark grey, the latter marking $\pl(A, <_\fl, <_\gl) = \{(1,1), (4,2), (5,6),(6,7),(9,9)\}$. \emph{Right:} The sparsity pattern of $LAR$,  for generic $A$.}
\label{fig:paretomatrix}
\end{figure}
\vspace{0.2cm}

\section{Homological persistence}
\label{sec:persistence}

An application of the $n$th homology functor to a nested sequence of chain complexes $C = (C_0 \su \cd \su C_L)$ in $\field$-${\mathtt{Vect}}$ yields a diagram of vector spaces
\[
\cd \to 0 \to H_n(C_0) \to \cd \to H_n(C_L) \to 0 \to \cd
\]
to which one may associate a graded $\field[t]$-module $\hl_n(C) = \oplus_i H_n(C_i)$, with $t$-action inherited from the induced map  $H_n(C_i) \to H_n(C_{i+1})$.  Modules of the form $\hl_n(C)$, often called  \emph{persistence modules} \cite{CarlssonTopology09,ZCComputing05}, are of marked interest in topological data analysis.

The  sequence $C_0 \su \cd \su C_L$ engenders a modular pair $(\fl, \gl)$ on the $n$th cycle space $Z_n(C_L)$, where
\begin{align}
\fl_i = \ker \partial_n(C_i);
\hspace{1cm} \gl_{i} = \im \partial_n(C_i);
\hspace{1cm} i = 1, \ld, L \notag;
\end{align}
and  $\gl_{L+1} = Z_n(C_L)$.  Each $n$-cycle $z$ maps naturally  to a homology class $ [z] \in H_n(C_i) $, where $i = \min \{j : z \in \fl_j\}$.  The proof of Proposition \ref{prop:generators} follows from commutativity of Figure \ref{fig:homologyfigure}.

\begin{proposition}
\label{prop:generators}
A basis $Z \in \bl(Z_n)$ is $(\chi_\fl, \chi_\gl)$-minimal if and only if the non-nullhomologous cycles,
$\{[z] : z \in Z, \; [z] \neq 0\}$, freely generate $\hl_n(C)$.
\end{proposition}

Computation of $(\chi_\fl, \chi_\gl)$-optimal bases requires slightly more than a rote application of Algorithm \ref{alg:paretoalgorithm} in general, since representations of the cycle space $Z_n$ seldom present \emph{a priori}.  Rather, the starting data is generally an indexed family $E_1, \ld, E_L$ of $\chi_{\fl}$-minimal bases for the chain groups of $C_L$ in all dimensions.  A linear $\chi_{\fl}$-order  on $E = E_1 \cup \cd \cup E_n$, denoted $<$, and a matrix representation of each boundary operator with respect to $E_n$, denoted $A_n$, constitute the input to Algorithm \ref{alg:chainalgorithm}.  As with Algorithm \ref{alg:paretoalgorithm_light}, we write $\pl(A_n)$ for $\pl(A_n, <, <)$ and $L_n$, $R_n$ for $L(A_n, <, <)$ and $R(A_n, <, <)$ respectively.  Where $n$ is not specified, each declaration should be understood for all $n$.

\begin{proposition}
Algorithm \ref{alg:chainalgorithm} terminates for some $t = t_0$.  The columns of
$$Z = L_{n+1}\inv R_n[ E_n, E_{n+1}-P_{n+1}]$$
form a $(\chi_\fl, \chi_\gl)$-minimal basis of $Z_n$, where $P_{n+1} = \{s: (s,t) \in P_{n+1}^{t_0}\}$.
\end{proposition}

\begin{proof}
Since $L_{n+1}\inv R_n$ is $\chi_\fl$-upper triangular, $Z$ is $\chi_\fl$-minimal in its closure, $Z_n$.  By Lemma \ref{lem:stillgotit}, the columns of $R_n\inv L_{n+1} \partial_{n+1} L_{n+2}\inv R_{n+1}$ contain a $\chi_\gl$-minimal basis of $cl(\partial_{n+1}) = \gl^n_L$, perforce the subset indexed by $P_{n+1}$.  In the $E_n$-standard representation of $C_n$ this basis corresponds to $L_{n+1}\inv R_n[E_n,P_{n+1}]$.  Any basis containing $L_{n+1}\inv R_n[E_n,P_{n+1}]$, in particular $Z$, is therefore $\gl$-minimal.
\end{proof}

\begin{figure}
\vspace{1cm}
\includegraphics[width=4in]{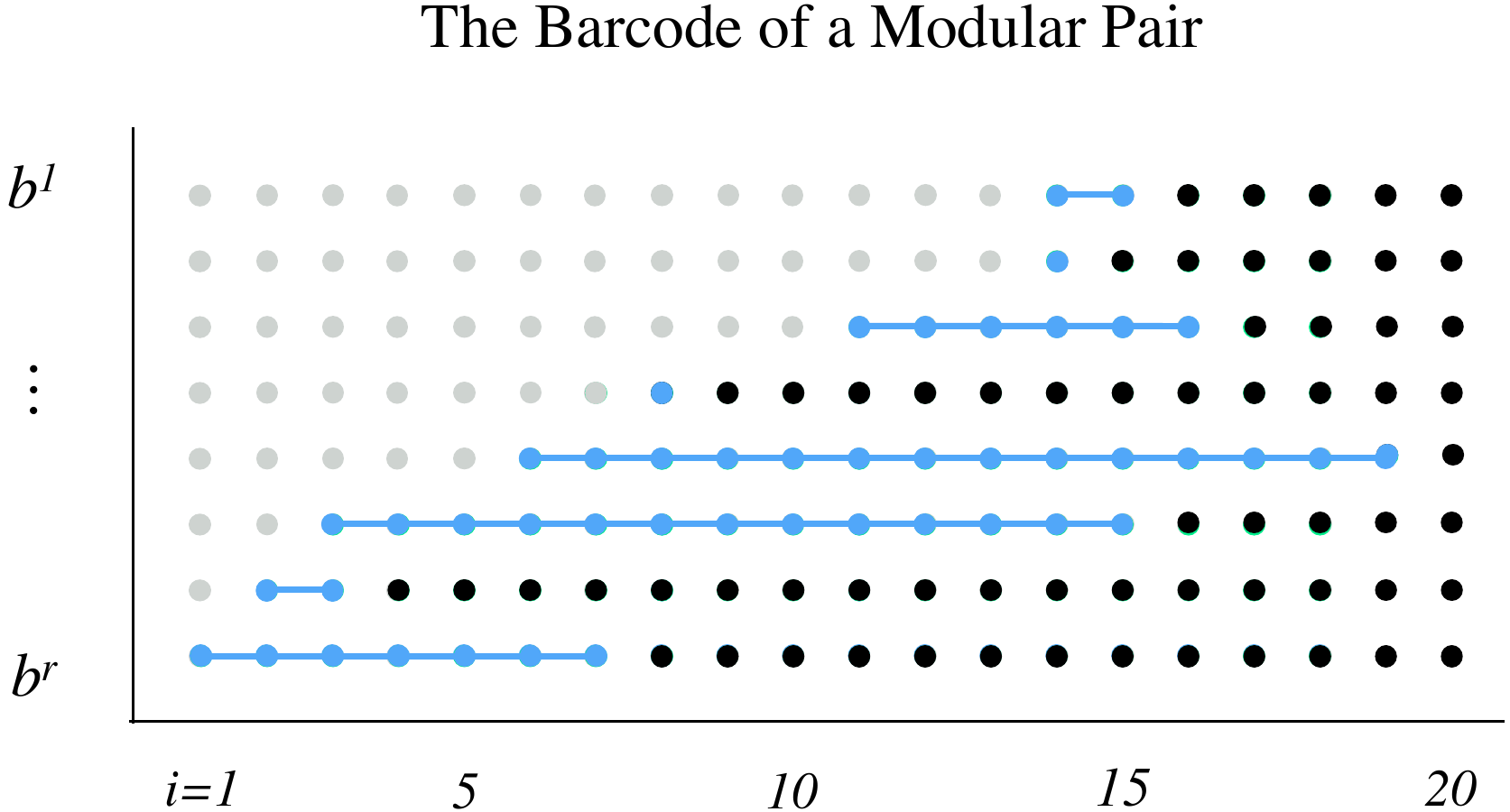} \hspace{.5cm}
\captionsetup{width=\textwidth}
\caption{Visual representation of a $(\chi_\fl,\chi_\gl)$-minimal basis $B = \{b^1,\ld, b^r\}$ in the special case where $\gl_i \su \fl_i$, $i = 1,2 \ld$.  Each row corresponds to an element of $B$.   The dot at location $(i,b)$ is black if $b \in \gl_i$, blue if $b \in \fl_i - \gl_i$, and grey otherwise.  The blue dots in column $i$ thus collectively represent a basis of $\fl_i/\gl_i$, the black dots represent zero elements, and the grey represent elements not contained in $\fl_i$.  The horizontal bars generated by the blue dots correspond exactly to the barcode of $(\fl,\gl)$, when $\fl$ and $\gl$ are the induced filtrations on $C_n$ for some filtered chain complex $C$.}
\label{fig:matroidbarcodegraphic}
\end{figure}
\vspace{0.2cm}

\begin{algorithm}
\caption{Chain reduction}
\label{alg:chainalgorithm}
\begin{algorithmic}[1]
\STATE Initialize $\pl_n^0 = \emptyset$, $\pl_n^1 = \pl(A_n)$, $t = 1$
 \WHILE{$\pl_{n}^t \neq \pl_{n}^{t-1}$ for some $n$}	
  		\STATE $t \leftarrow t+1$
    			\STATE $A_n\leftarrow L_n A_n L_{n+1}\inv$
	    		\STATE $\pl_n^t = \pl(A_n)$		
 \ENDWHILE
			\STATE $A_n\leftarrow R_{n-1}\inv A R_{n}$
\end{algorithmic}
\end{algorithm}

\section{Acyclic Relations}
\label{sec:acyclicrelations}

Consider a nested sequence of cellular spaces  $X_0 \su \cd \su X_L  = X$.  Most such filtered complexes that arise in scientific applications have very large $n$-skeleta, even when $n$ is small.  Their boundary operators are highly sparse, and become accessible to computation only through sparse representation in memory.  One drawback of sparse representation, however, is an oft disproportionate increase in the cost of computing matrix products.   All things being equal, it is therefore preferable to execute fewer iterations of Algorithm \ref{alg:chainalgorithm} than many when $A$ is sparsely represented.  Since work stops when $|P_n^t| = \tm{rank} A_n$, one might naively hope to see improved performance when $P_n^1$ is close to $\tm{rank} A_n$ -- a hope realized in practice.

Since $P_n^t$ is entirely determined by $<$, anyone looking to maximize $|P_n^1|$ will do so by an informed choice of linear order.   Enter acyclic relations, which provide a means to determine  \emph{a priori} the inclusion of certain sets, called acyclic matchings, in $P_n^1$.  Our strategy will be to find a favorable acyclic matching, and from this to engineer a compatible order $<$.   Much effort has already been invested in the design and application of matchings in algebra, topology, combinatorics, and computation via the discrete Morse Theory of  R.\ Forman and subsequent literature \cite{FormanMorse98,KozlovDiscrete05,MNMorse13,NTTDiscrete13,SkoeldbergMorse06}.  We will discuss the details of our own approach to matchings, and its connection to discrete Morse theory in \cite{henselmanMFA16}.  At present we limit ourselves to a description of its use in Algorithm \ref{alg:chainalgorithm}.

A binary relation $R$ on a ground set  $E$ is \emph{acyclic} if the transitive closure of $R$ is antisymmetric.   Evidently, $R$ is acyclic if and only if the transitive closure of $R \cup \ad(E \times E)$ is a partial order on $E$.  The following observation is similarly clear, but bears record for ease of reference.

\begin{lemma}
\label{lem:clear}
If $E$ is finite and $f$ is a real-valued function on $E$, then every acyclic $f$-relation  extends to an $f$-linear order.
\end{lemma}
\begin{proof}
Assume for convenience $\im f = \{1,\ld, n\}$.  Fix an acyclic $f$-relation $R$ and for $i = 1, \ld, n$ let $\al_i$ be a linear order on $f \inv (i)$ respecting the transitive closure of  $R \cup \ad(E \times E)$.  The  closure of
$$
\{(s,t): f(s)<f(t)\} \cup \al_1 \cup \cd \cup \al_n
$$
is a linear order of the desired form.
\end{proof}

Suppose now that $K$ is a finite-dimensional, $\field$-linear chain complex supported on $\{1,\ld, N\}$, that $E_n$ freely generates $K_n$, and that $\deg(e) = n$ for each $e \in E_n$.  If $A_n$ is the matrix representation of $\partial_n$ with respect to this basis, then a \emph{matching} on $(K,E)$ is a subrelation
$$V \su \tm{supp }A_1 \cup \cd \cup \tm{supp }A_N$$
such that for each $s_0$ and $t_0$ in $E$,   the pairs $(s_0, t)$ and $(s,t_0)$ belong to $V$ for at most one $s$ and one $t$.  ``Flipping'' $V$ yields a relation
$$
R_V = \le_{\deg} - V \cup V\op
$$
where $\le_{\tm{deg}} = \{(s,t): \tm{deg}(s)<\tm{deg}(t)\}$.  We say $V$ is \emph{acyclic} ($f$-acyclic) if $R_V$ is acyclic ($f$-acyclic).   Taken together, Proposition \ref{prop:includes} and Lemma \ref{lem:clear} imply that every $f$-acyclic matching includes  into $\pl(A_n, <)$ for some $f$-linear $<$.   In particular, given large $V$, one can always find large $P_n^1$.

\begin{proposition}
\label{prop:includes}
If $V$ is a matching and $<\op$ linearizes $R_V$, then $V$ includes into $\cup_n\pl(A_n,<)$.
\end{proposition}

It is quite easy to check wether $V$ is acyclic in practice.  The task of deciding $f$-acyclisity, though somewhat more involved in general, reduces to a linear search when $f$ is the characteristic function of a filtration.

\begin{lemma}
An acyclic matching $V$ is $\chi_\fl$-acyclic if and only if $\chi_\fl(s) = \chi_\fl(t)$ for all $(s,t) \in V$.
\end{lemma}

We close with a brief but computationally useful observation concerning the order of operations whereby products of form $L_nAL_{n+1}$ are computed in Algorithm \ref{alg:chainalgorithm}.  Elementary calculations show that $A_nL_{n+1}$ the columns of $A_nL_{n+1}$ indexed by $P_{n+1}^t$ vanish at each step of the process.  Thus the problem of computing $L_nA_nL_{n+1}$ reduces to that of computing $L_n A[\; E_{n-1} \;, \;E_n - P_{n}^t\;]$.  In particular, if only free generators for the first $N-1$ homology groups are desired, then while there is no need to compute $P_m$ for  $m> N$, identifying a subset $S$ of $P_{N+1}$ allows one to reduce the calculation of $L_N A_N$ to that of  $L_N A_N[\; E_N \;, E_{N+1} - S\;]$.

Top-dimensional boundary operators have special status in scientific computation, as they are often the  the largest by far.  If $S$ may be determined formulaically prior to the construction and storage of $A_{N}$ --  for example, by determining a closed-form expression for an acyclic matching -- then the cost of generating large portions of this matrix may be avoided altogether.  As reported in the Experiments section, the effect of this reduction may be to drop the memory-cost of computation by several orders of magnitude.  To our knowledge, this is the first principled use of acyclic matchings to avoid the construction not only of large portions of the cellular boundary operator, but of the underlying complex itself.

\section{Experiments}
\label{sec:experiments}

An instance of Algorithm \ref{alg:chainalgorithm} incorporating the optimization described in \S\ref{sec:acyclicrelations} has been implemented in the \Eirene  library for homological algebra \cite{EIRENE}.  Experiments were conducted on a personal computer with Intel Core i7 processor at 2.3GHz, with 4 cores, 6MB of L3 Cache, and 16GB of RAM.  Each core has 256 KB of L2 Cache.   Results for a number of sample spaces, including those appearing in recent published benchmarks, are reported in Tables 1, 2, and 3. All homologies are computed using $\Z_2$ coefficients.

Our first round of experiments computes persistent homology of a Vietoris-Rips filtration on a random point cloud on $n$ vertices in $\R^{20}$, for values of $n$ up to 240. Persistent homology with representative generators is computed in dimensions 1, 2, and 3, with the total elapsed time and memory load (heap) recorded in Table \ref{tab:RG}.

\begin{table}
\begin{center}
\begin{tabular}{ llllll}
\hline
$\#$ Vertices  & Size (B)  & Time (s) & Heap(GB) & CR  \\ \hline
40 &  0.00 &2.13 & 0.02&0.14 \\
80& 0.00 &4.36 & 0.02 &0.07\\
120& 0.19 &15.7 & 4.58&0.04\\
160&0.82 & 44.1 &21.65&0.03\\
200 &2.54 &124&33.46 &0.03\\
240 & 6.36 &407& 53.18&0.02 \\ \hline
\end{tabular}
\end{center}
\vspace{0.5cm}
\captionsetup{width=0.8\textwidth}
\caption{Persistence with generators in dimensions 1, 2, and 3 for the Vietoris-Rips complex of point clouds sampled  from the uniform distribution on the unit cube in $\R^{20}$.  \emph{Size} refers to the size of the 4-skeleton of the complete simplex on $n$ vertices.  \emph{CR}, or {compression ratio}, is the quotient of the size of the generated subcomplex by the 4-skeleton of the underlying space. }
\label{tab:RG}
\end{table}

Our second round of experiments parallels the benchmarks published in Fall 2015 \cite{OPT+roadmap}. Note that Tables 6.1 and 6.2 of this reference record time and space expenditures for ceratin large point clouds on various publicly available software packages, some of which are run on a cluster. We append one new example to this table: RG1E4, a randomly generated point cloud of 10,000 points in $\R^{20}$. This complex, the 2-skeleton of which has over 160 billion simplices, is, at the time of this writing, the largest complex whose homology we have computed. All the instances in Table \ref{tab:benchmarks} record computation of persistent $H_1$.

\begin{table}
\begin{center}
\begin{tabular}{ lllll}
\hline
VR Complex & Size (M) & Time (s) & Heap (GB)  \\ \hline
C.\ elegans & 4.37 & 1.33 & 0.00 \\
Klein  & 10.1 & 1.76 & 0.01 \\
HIV &  214 & 12.6 & 0.04 \\
Dragon 1 &166 & 15.6 & 0.08   \\
Dragon 2 & 1.3k & 141 & 2.32   \\
RG1E4 & 1.66k & 3.12k & 35.7   \\ \hline
\end{tabular}
\end{center}
\vspace{0.5cm}
\captionsetup{width=0.8\textwidth}
\caption{One-dimensional persistence for various Vietoris-Rips complexes.  The data for C. elegans, Klein, HIV, Dragon 1, and Dragon 2 were drawn directly from the published benchmarks in \cite{OPT+roadmap}.    RG1E4 is a random geometric complex on $10^4$ vertices, sampled  from the uniform distribution on the unit cube in $\R^{20}$. }
\label{tab:benchmarks}
\end{table}

The third set of experiments shows where the algorithm encounters difficulties. We compute higher dimensional homology of fixed (non-filtered) complexes arising from combinatorics. These complexes -- the {\em matching} and {\em chessboard} complexes -- are notoriously difficult to work with, as they have very few vertices, very many higher-dimensional simplices, and relatively large homology \cite{SWTorsion07}. The performance of \Eirene in Table \ref{tab:stress} is consistent with the expected difficulty: acyclic compression can only do so much for such complexes.

\begin{table}
\begin{center}
\begin{tabular}{ lllll}
\hline
Complex &Dim & Size (M) & Time (s) & Heap (GB)  \\ \hline
Chessboard  & 7 & 1.44 & 8.38k & 26.9 \\
Matching  &3& 0.42 & 9.98k & 37.0   \\\hline
\end{tabular}
\end{center}
\vspace{0.5cm}
\captionsetup{width=0.8\textwidth}
\caption{Homology of two unfiltered spaces, the chessboard complex $C_{8,8}$ and the matching complex $M_{3,13}$.}
\label{tab:stress}
\end{table}

Finally, Figure \ref{fig:4simplices} illustrates the degree of compression achieved through \Eirene in the context of a random point cloud in dimension 20. The horizontal axis records the number of vertices, $n$, used in the complex.  The vertical axis records a ratio of size of various quantities relative to the rank of the boundary operator $\partial_4$ of the complete simplex on $n$ vertices.

\begin{figure}[h]
\vspace{0.7cm}
  \centering
  \includegraphics[width=4in]{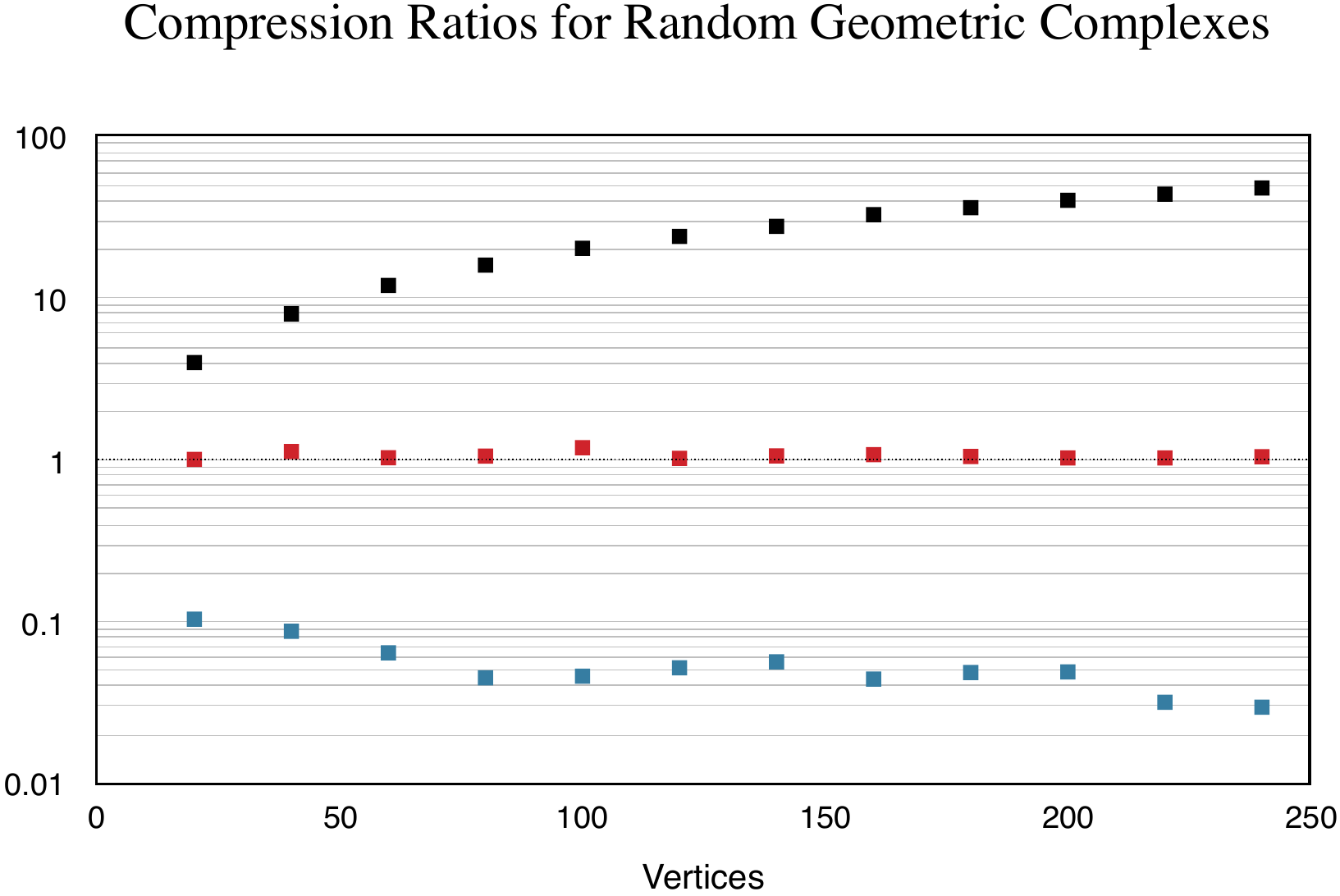}\hspace{.7cm}
 \vspace{0.4cm}
 \captionsetup{width=\textwidth}
  \caption{Persistent $H_3$ for a family of random geometric complexes.   Samples of cardinality $k$ were drawn from the uniform distribution on the unit cube in $\R^{20}$,  $k = 20, \ld, 240$.  The method of \S6 was applied to the distance matrix $d$ of each sample, resulting in a morse complex $M$.   Recall that the n-cells of $M$ are indexed by $M_n = E_n - P^1_n - P^1_{n+1}$, where $E_n$ is the family of $n$-faces of the simplex on $k$ vertices, and $P^1$ is an acyclic matching.  The \Eirene{} library applies a dynamic construction subroutine to build $M_n$ from $d$ directly.  This subroutine {generates} the elements of $X_n = E_n - P^1_{n-1}$ sequentially and stores the elements of $M_n$ in memory;\ it does not generate elements of $E_n-X_n$, nor does it store any combinatorial $n$-simplex in memory.   In the figure above, black, red, and blue correspond to the ratios of $|E_4|$, $|X_4|$, and $|M_4|$, respectively, to $\tm{dim }\partial_4(E_4)$.  All statistics represent averages taken across 10 samples.}
    \label{fig:4simplices}
\end{figure}

\begin{figure}[h]
\vspace{0.7cm}
  \centering
  \includegraphics[width=2.5in]{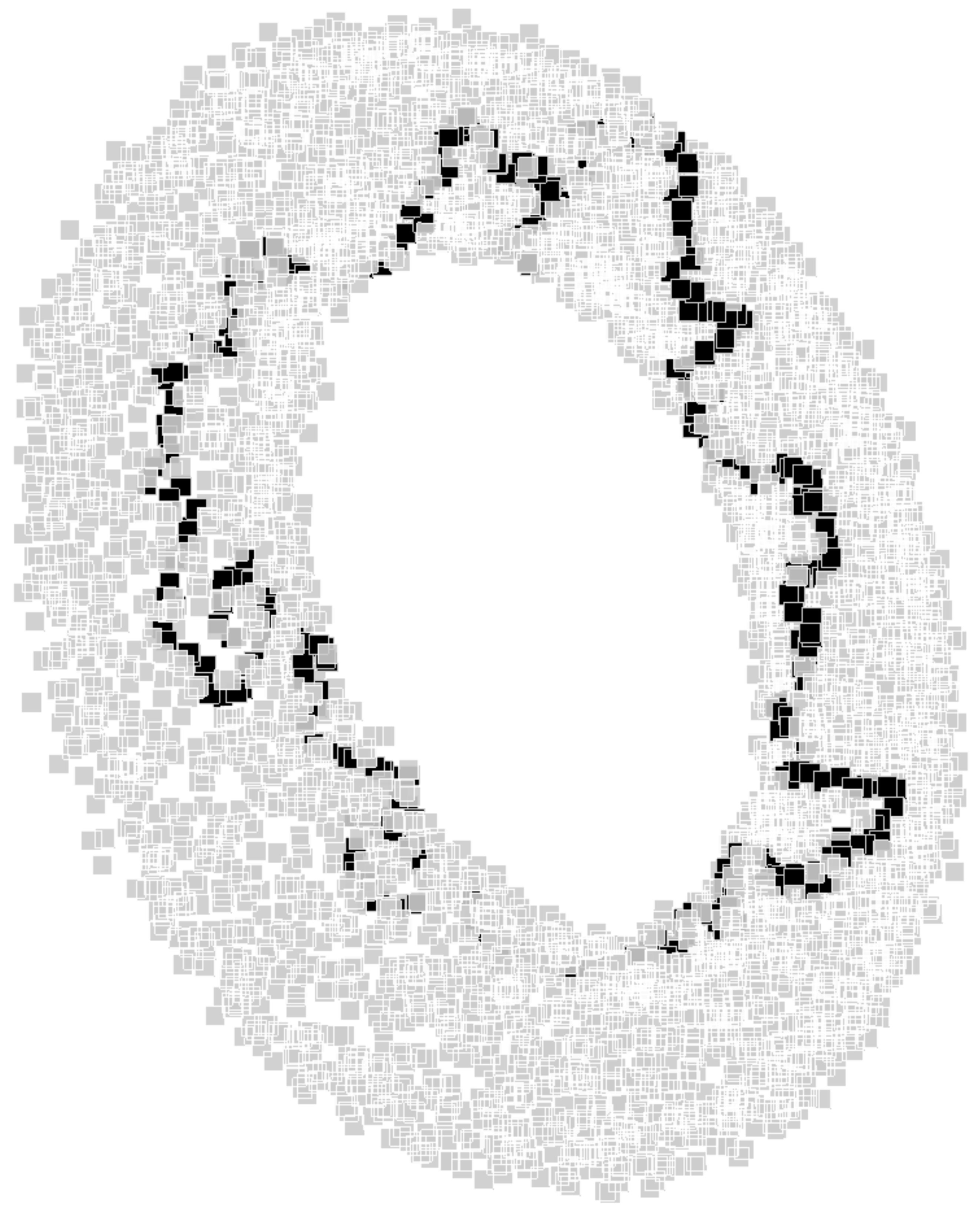}\hspace{.7cm}
 \vspace{0.4cm}
 \captionsetup{width=\textwidth}
  \caption{A one-dimensional class representative.   Grey points represent a sample of $5 \times 10^{3}$ points drawn from a torus embedded in $\R^3$, with uniform random noise.  Free generators for the associated persistence module, thresholded at three times the maximum noise level, were computed with the \Eirene library for homological algebra.  A representative for the unique 1-dimensional class that survived to infinity was plotted with the open-source visualization library \emph{Plotly} \cite{plotly15}.   Vertices incident to the cycle representative appear in black.}
    \label{fig:4simplices}
\end{figure}

\newcommand{\T}{\ar[ddd] \ar[rrr] \ar[dr]}
\newcommand{\Fs}{\ar[rrr] \ar[ddd]}
\newcommand{\Df}{\ar[ddd] \ar[dr]}
\newcommand{\Rf}{\ar[rrr] \ar[dr]}
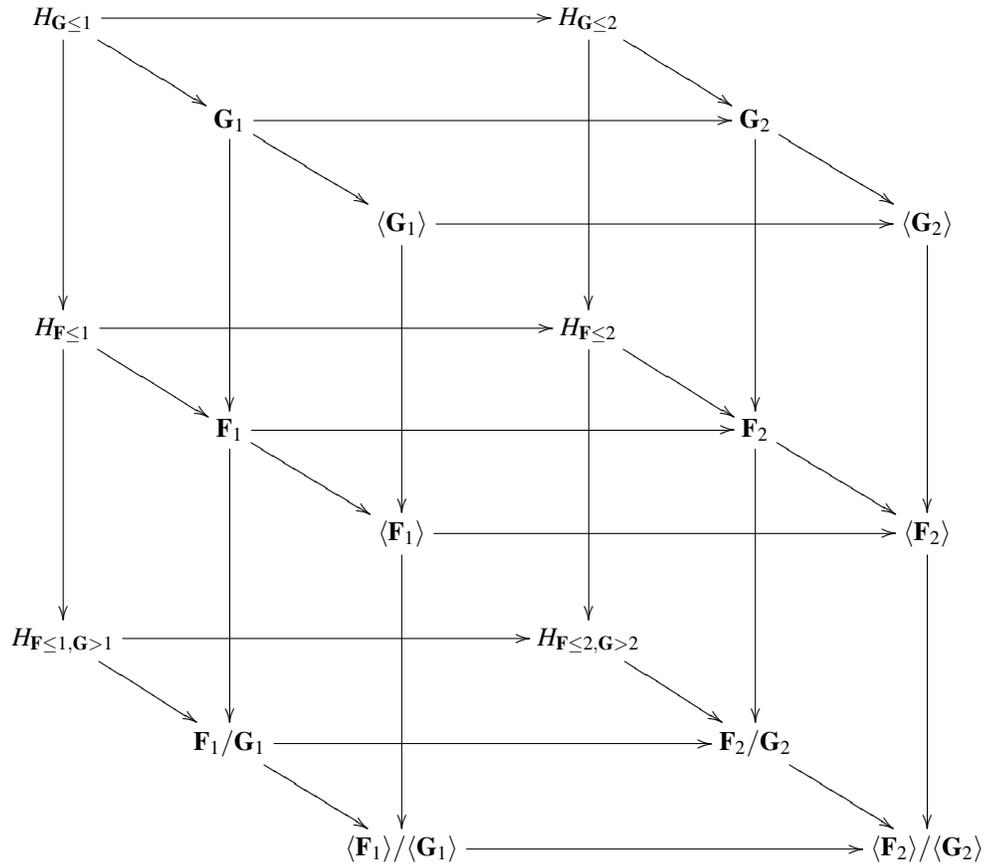
\begin{figure}[p!]
\xymatrix{
		H_{\gl\le1}\T&&&H_{\gl\le2} \Df &	&&	&&	\\		 			
		&\gl_1\T&&&\gl_2\Df&&&\\									
		&&\nr{\gl_1} \Fs &&&\nr{\gl_2}\ar[ddd]	&&\\					
		H_{\fl\le1}\T&&&H_{\fl \le 2}\Df&&&&\\						
		&\fl_1\T	&&	&\fl_2\Df&&\\									
		&&\nr{\fl_1}\Fs	&&	&\nr{\fl_2}\ar[ddd]&	\\					
		H_{\fl\le1,\gl>1}\Rf&&&H_{\fl\le2,\gl>2}\ar[dr]&&&&&\\			
		&\fl_1/\gl_1\Rf&&&\fl_2/\gl_2\ar[dr]&&&&\\					
		&&\nr{\fl_1}/\nr{\gl_1}\ar[rrr]&&&\nr{\fl_2}/\nr{\gl_2}&&&\\		
		&&&&&&&&\\												
}
\captionsetup{width=\textwidth}
\caption{A commutative diagram in $\mathtt{Set}$.  All maps preserve zero-elements.  Recall that $H = B \cup \{0\} \su \field^r$ for some $(\chi_\fl, \chi_\gl)$-optimal basis $B$.  Arrows between the top 12 objects are inclusions.  The lowest two vertical arrows are  quotient maps in $\field$-$\mathtt{Vect}$.  The restriction of an arrow $a:X \to Y$ to codomain $X-\{0\}$ is an inclusion of matroid bases when $a$ is oblique and $X \su H$;  it is an inclusion of vector bases when $a$ is the composition of colinear oblique arrows.  The diagram may be extended arbitrarily far to the right.}
\label{fig:homologyfigure}
\end{figure}

\section*{Acknowledgements}
\label{sec:ack}

The authors wish to express their sincere gratitude to C.\ Giusti and V.\ Nanda for their encyclopedic knowledge, useful comments, and continued encouragement.  This work is supported by US DoD contracts FA9550-12-1-0416, FA9550-14-1-0012, and NO0014-16-1-2010.

\bibliographystyle{acm}
\bibliography{Eirene.bib}{}

\end{document}